\documentclass[reqno,twoside]{amsart}

\usepackage{amsfonts,amsmath,amssymb}
\usepackage{mathrsfs,mathtools}
\usepackage{enumerate}
\usepackage{hyperref}
\usepackage{esint}
\usepackage{graphicx}
\usepackage{bm}
\usepackage{commath}
\DeclareMathAlphabet{\mathpzc}{OT1}{pzc}{m}{it}

\usepackage{multirow}

\usepackage{floatrow}

\usepackage{float}

\usepackage{placeins}

\usepackage{flafter} 
\usepackage[textsize=small]{todonotes}
\usepackage{caption}

\numberwithin{equation}{section}

\usepackage[dvips,bottom=1.4in,right=1in,top=1in, left=1in]{geometry}

\usepackage[utf8]{inputenc}
\usepackage[english]{babel}
\usepackage{array}
\usepackage[capitalize]{cleveref}
\usepackage{enumitem}

\makeatletter
\def\eqnarray{\stepcounter{equation}\let\@currentlabel=\theequation
\global\@eqnswtrue
\tabskip\@centering\let\\=\@eqncr
$$\halign to \displaywidth\bgroup\hfil\global\@eqcnt\z@
  $\displaystyle\tabskip\z@{##}$&\global\@eqcnt\@ne
  \hfil$\displaystyle{{}##{}}$\hfil
  &\global\@eqcnt\tw@ $\displaystyle{##}$\hfil
  \tabskip\@centering&\llap{##}\tabskip\z@\cr}

\def\endeqnarray{\@@eqncr\egroup
      \global\advance\c@equation\m@ne$$\global\@ignoretrue}

\def\@yeqncr{\@ifnextchar [{\@xeqncr}{\@xeqncr[1pt]}}
\makeatother

\DeclareMathAlphabet\gothic{U}{euf}{m}{n}

\newtheorem{theorem}{Theorem}[section]

\newtheorem{definition}[theorem]{Definition}

\newtheorem{proposition}[theorem]{Proposition}

\newtheorem{remark}[theorem]{Remark}
\numberwithin{equation}{section}

\numberwithin{equation}{section}

\def\RR{{\mathbb{R}}}
\def\NN{{\mathbb{N}}}

\def\CC{{\mathbb{C}}}
\def\ZZ{{\mathbb{Z}}}

\title[Discrete evolution equations]{Convergence of solutions of discrete semi-linear space-time fractional evolution equations}

\author{Harbir Antil}
\address{H. Antil, Department of Mathematical Sciences, George Mason University, Fairfax, VA 22030, USA.}
\email{hantil@gmu.edu}

\author{Carlos Lizama}
\address{C. Lizama, Departamento de Matem\'atica, Universidad de Santiago de Chile, Casilla 307- Correo 2, Santiago, Chile.}
\email{carlos.lizama@usach.cl }

\author{Rodrigo Ponce}
\address{R. Ponce, Instituto de Matem\'atica y Fisica, Universidad de Talca, Casilla 747, Talca, Chile.}
\email{rponce@inst-mat.utalca.cl }

\author{Mahamadi Warma}
\address{M. Warma, University of Puerto Rico, Rio Piedras Campus, Department of Mathematics,  College of Natural Sciences,  17 University AVE. STE 1701  San Juan PR 00925-2537 (USA). }
\email{mahamadi.warma1@upr.edu, mjwarma@gmail.com}

\thanks{The work of HA is partially supported by the NSF grants DMS-1818772,
DMS-1913004 and the Air Force Office of Scientific Research under Award NO: FA9550-19-1-0036.}

\thanks{The work of CL is partially supported by the Conicyt under Fondecyt grant N0: 1180041}

\thanks{The work of MW is partially supported by the Air Force Office of Scientific Research under Award NO:  FA9550-18-1-0242}



\begin{document}

\begin{abstract}
Let $(-\Delta)_c^s$ be the realization of the fractional Laplace operator on the space of continuous functions $C_0(\RR)$, and let $(-\Delta_h)^s$ denote the discrete fractional Laplacian on $C_0(\ZZ_h)$, where $0<s<1$ and $\ZZ_h:=\{hj:\; j\in\ZZ\}$ is a mesh of fixed size $h>0$.
We show that solutions of fractional order semi-linear Cauchy problems  associated with the discrete operator $(-\Delta_h)^s$ on $C_0(\ZZ_h)$ converge to solutions of the corresponding Cauchy problems associated with the continuous operator $(-\Delta)_c^s$. In addition, we obtain that the convergence is uniform in $t$ in compact subsets of $[0,\infty)$. We also provide  numerical simulations that support our theoretical results.
\end{abstract}

\maketitle

\section{Introduction, notations and main results}

Fractional order operators have recently emerged as a modeling alternative in various branches of science and technology. In fact, in many situations, the fractional models reflect better the behavior of the system both in the deterministic and stochastic contexts.
A number of stochastic models for explaining anomalous diffusion have been
introduced in the literature; among them we mention  the fractional Brownian motion; the continuous time random walk;  the L\'evy flights; the Schneider grey Brownian motion; and more generally, random walk models based on evolution equations of single and distributed fractional order in  space (see e.g. \cite{DS,GR,Man,Sch}).  In general, a fractional diffusion operator corresponds to a diverging jump length variance in the random walk. Further applications include, imaging \cite{antil2017spectral, HAntil_CNRautenberg_2019a}, machine learning \cite{antil2019bilevel}, and geophysics \cite{CWeiss_BvBWaanders_HAntil_2018a}. We refer to \cite{Val,Gal-War,GW-CPDE} and the references therein for a complete analysis, the derivation and more applications of fractional order operators.

Numerical methods for fractional PDEs have recently received a great deal of attention, see \cite{JLZ18} and references therein. The study of convergence of discrete solutions  to continuous  ones for the associated {\it stationary} linear problem, that is, the time independent problem
\begin{align}\label{SP}
(-\Delta)^su=f \;\mbox{ in }\;\RR\;\;\mbox{ and }\; (-\Delta_h)^su_h=f  \;\mbox{ in }\;\ZZ_h,
\end{align}
where $0<s<1$ is a real number, $(-\Delta)^s$ denotes the fractional Laplace operator defined on the real line $\RR$ and,  for a fixed real number $h>0$, the operator $(-\Delta_h)^s$ is the discrete fractional Laplace operator  and $\ZZ_h:=\{hj:\; j\in\ZZ\}$ is a mesh of fixed size $h>0$, and $f$ does not depends on $u$ (resp. $u_h$), has been first and completely analyzed by Ciaurri et al \cite{Setal},  where the authors have proved the convergence of solutions in  $\ell^\infty$-spaces. They have also obtained explicit convergence rates.

 However, the study of the convergence of discrete to continuous solutions for the  non-stationary i.e., the time-dependent problem, is completely open.

The main concern of the present paper is to solve this open problem. We  show the convergence of solutions of a class of discrete space-time fractional evolution equations to the solutions of the corresponding equations associated with the continuous operator. More precisely, we consider the following two systems:
\begin{equation}\label{EE1}
\begin{cases}
\mathbb D_t^\alpha u+(-\Delta)^su=f(u)\;\;&\mbox{ in }\;(0,\infty)\times\RR,\\
u(0,\cdot)=u^0 &\mbox{ in }\;\RR,
\end{cases}
\end{equation}
and
\begin{equation}\label{EE2}
\begin{cases}
\mathbb D_t^\alpha u_h+(-\Delta_h)^su_h=f(u_h)\;\;&\mbox{ in }\;(0,\infty)\times \ZZ_h,\\
u_h(0,\cdot)=u_h^0&\mbox{ in }\;\ZZ_h,
\end{cases}
\end{equation}
where in \eqref{EE1}-\eqref{EE2}, $0<\alpha\le 1$ and $0<s<1$ are real numbers,  $\mathbb D_t^\alpha$ denotes the Caputo time-fractional derivative of order $\alpha$  and $f$ is a nonlinear function satisfying a local Lipschitz continuity condition. Notice that if $\alpha=1$, the above systems correspond to the semi-linear continuous and discrete heat equations, respectively.



  Our main result (Theorem \ref{main-theorem1}) states that if $0<s<1$, $0<\alpha\le 1$, $f\in C^1(\RR)$ with $f(0)=0$, then for every $u_h^0\in C_0(\ZZ_h)$, there exist $u^0\in C_0(\RR)$  and $\mathcal T>0$, such that the corresponding solutions $u$ and $u_h$ of the systems \eqref{EE1} and \eqref{EE2}, respectively, satisfy
\begin{align}\label{eee}
\lim_{h\downarrow 0}\sup_{0\le t\le\mathcal T}\|\mathcal R_hu(t)-u_h(t)\|_{C_0(\ZZ_h)}=0,
\end{align}
where $\mathcal R_h$ is a suitable mapping from $C_0(\RR)$ to $C_0(\ZZ_h)$ (see \eqref{mapP} below). In particular, we obtain that for $u_0\in C_0(\RR)$, then taking $u_h^0=\mathcal R_hu^0$, we have that the corresponding solutions also satisfy the convergence \eqref{eee}. The main tools we shall use are the convergence results for the stationary problem \eqref{SP} (see Theorem \ref{thm-Stinga}) and a suitable version of the Trotter-Kato approximation theorem.

We observe that
the existing theoretical error estimates are largely limited to either energy type norms or to the $L^2$-norm, see for instance, \cite{GA:JP15} where error estimates for finite element discretization has been carried out. There are no
existing pointwise error estimates for parabolic fractional PDEs, even though from a practical point of view, it is much
easier to observe the pointwise values. In fact, the techniques for pointwise error estimates are
significantly different than in the other norms \cite{SCBrenner_RLScott_1994a}. It is unclear how to extend such arguments to nonlocal/fractional PDEs since solutions of such PDEs do not enjoy enough regularity as the local PDEs, and therefore makes the numerical analysis of such equations very challenging. There are many papers where the authors have assumed that such PDEs have very smooth solutions in order to do the numerical analysis. We refrain from citing those works, but we point out that solutions of fractional PDEs involving $(-\Delta)^s$ do not have enough regularity as in the classical case $s=1$. We refer to \cite{RS-DP2,RS-DP} for more details on this topic.

The rest of the paper is structured as follows. In Section \ref{sec-not} we give some notations and introduce the function spaces needed to investigate our problem. We state the main results of the paper in Section \ref{sec-mr-state} and give some preliminary results in Section \ref{sec-pre} as they are needed throughout the paper. Section \ref{sec-pro} contains the proofs of our main results. Finally in Section \ref{numerics} we present some numerical simulations that confirm our theoretical findings.

\subsection{Notations}\label{sec-not}
In this section we fix some notations and introduce the function spaces needed to investigate our problem.
Let $0<s<1$ be a real number and $N\in\NN$. The continuous fractional Laplace operator in $\RR^N$ is defined by the following singular integral:
\begin{align}\label{FL}
(-\Delta)^su(x):=&C_{N,s}\mbox{P.V.}\int_{\RR}\frac{u(x)-u(y)}{|x-y|^{N+2s}}\;dy\notag\\
=&C_{N,s}\lim_{\varepsilon\downarrow 0}\int_{\{y\in\RR^N:\; |x-y|>\varepsilon\}}\frac{u(x)-u(y)}{|x-y|^{N+2s}}\;dy, \qquad x\in\RR,
\end{align}
provided that the limit exists for a.e. $x\in\RR^N$. The normalization constant $C_{N,s}$ is given by
\begin{align*}
C_{N,s}:=\frac{s2^{2s}\Gamma\left(\frac N2+s\right)}{\pi^{\frac N2}\Gamma(1-s)}.
\end{align*}
If $N=1$, then we shall denote $C_s:=C_{1,s}$.
We refer to \cite{Caf3,Grub,Grub2,RS-DP2,War} and the references therein for the class of functions for which the limit  in \eqref{FL} exists.


We consider a mesh of fixed side $h>0$ in $\RR$ given by
$$\ZZ_h:=\{hj:\;j\in \ZZ\}.$$
For a function $u:\ZZ_h\to\RR$ we define the discrete Laplace operator $\Delta_h$ by
\begin{align*}
-\Delta_hu(jh):=-\frac{u((j+1)h)-2u(jh)+u((j-1)h)}{h^2}.
\end{align*}
It is well-known that the operator $\Delta_h$ generates a strongly continuous submarkovian semigroup $(e^{-t\Delta_h})_{t\ge 0}$ on $\ell^p(\ZZ_h)$ ($1\le p\le\infty$) given by
\begin{align}\label{SG}
e^{t\Delta_h}f(jh)=e^{-\frac{2t}{h^2}}\sum_{n\in\ZZ}I_{n-j}\left(\frac{2 t}{h^2}\right)f(nh),\qquad t\ge 0,
\end{align}
where $I_\nu$ denotes the modified Bessel function of the first kind and of order $\nu$ (see e.g. \cite[Theorem 1.1 and Remark 1]{LR18} and the references therein).

For $0<s<1$, we define the fractional powers $(-\Delta_h)^s$ of $-\Delta_h$ on $\ZZ_h$ by
\begin{align}\label{dFL}
(-\Delta_h)^su(jh):=\frac{1}{\Gamma(-s)}\int_0^\infty \left(e^{t\Delta_h}u(jh)-u(jh)\right)\frac{dt}{t^{1+s}},
\end{align}
where we recall that $(e^{t\Delta_h})_{t\ge 0}$ is the submarkovian semigroup  given in \eqref{SG} and $\Gamma(-s)=-\frac{\Gamma(1-s)}{s}$.

Let
\begin{align*}
\ell_s:=\Big\{u:\ZZ_h\to\RR\mbox{ such that } \sum_{m\in\ZZ}\frac{|u(mh)|}{(1+|m|^{1+2s})}<\infty\Big\}.
\end{align*}
It has been shown in \cite[Theorem 1.1]{Setal} that if $u\in \ell_s$, then
\begin{align}\label{eq:DeltahR}
(-\Delta_h)^su(jh)=\sum_{m\in\ZZ, m\ne j}\Big(u(jh)-u(mh)\Big)R^s_h(j-m),
\end{align}
where the kernel $R^s_h$ is given by
\begin{equation*}
R^s_h(m):=
\begin{cases}
\displaystyle \frac{4^s\Gamma(\frac 12+s)}{\sqrt{\pi}|\Gamma(-s)|}\cdot\frac{\Gamma(|m|-s)}{h^{2s}\Gamma(|m|+1+s)},\;\;\;&\mbox{ if }\;m\in\ZZ\setminus\{0\},\\
=0&\mbox{ if }\; m=0.
\end{cases}
\end{equation*}
Equivalently, we can derive another pointwise
explicit formula of  $(-\Delta_h)^s$ as follows: (see e.g. \cite[Section 3.1]{LR18})
\begin{equation}\label{eq:DeltahK}
(-\Delta_h)^su(jh)= \sum_{n\in \ZZ} K^{s}_h(j-n)u(nh),
\end{equation}
where
$$
K^{s}_h(n) := \frac{(-1)^n \Gamma(2s+1)}{\Gamma(1+s+n)\Gamma(1+s-n) h^{2s}}.
$$
We notice that $K^{s}_h(n)= -R^{s}_h(n)$  for all $n\in \ZZ \setminus \{0\}$ and each $0<s<1$ (see e.g. \cite[Proposition 1]{LR18}).

Next, we let the Banach space
\begin{align*}
C_0(\ZZ_h):=\Big\{f:\ZZ_h\to\RR\;\mbox{ with }\;\lim_{|jh|\to\infty}f(jh)=0\Big\},
\end{align*}
be endowed with the sup norm
\begin{align*}
\|f\|_{C_0(\ZZ_h)}:=\sup_{j\in\ZZ}|f(jh)|.
\end{align*}
We also define $C_0(\RR)$ as
\begin{align*}
C_0(\RR):=\Big\{f:\;f\in C(\RR)\cap L^\infty(\RR)\;\mbox{ and }\;\lim_{|x|\to\infty}f(x)=0\Big\}.\,
\end{align*}
and we endow it with the sup norm.

If $0<\gamma\le 1$ and $k\in \NN\cup\{0\}$, we define the space of H\"older continuous functions
\begin{align*}
C^{\gamma,k}(\RR):=\Big\{U\in C(\RR):\; [U^{(k)}]_{C^{0,\gamma}(\RR)}:=\sup_{x,y\in\RR, x\ne y}\frac{|U^{(k)}(x)-U^{(k)}(y)|}{|x-y|^{\gamma}}<\infty\Big\}.
\end{align*}
For a real number $h>0$, we define the bounded linear map
\begin{align}\label{mapP}
\mathcal R_h: C_0(\RR)\to C_0(\ZZ_h),\; u\mapsto \mathcal R_hu:\ZZ_h \to\RR,\; (\mathcal R_hu)(jh)=u(jh).
\end{align}

Next, let $X$ be a Banach space.
For a fixed real number $0<\alpha\le 1$ and a function $u\in C([0,T);X)$, we define the Caputo time-fractional derivative as follows:
\begin{align}\label{CFD}
\mathbb D_t^\alpha u(t):=\frac{d}{dt}\int_0^t\frac{(t-\tau)^{-\alpha}}{\Gamma(1-\alpha)}\Big(u(\tau)-u(0)\Big)\;d\tau.
\end{align}
We notice that if $u$ is smooth, then \eqref{CFD} is equivalent to the classical definition of $\mathbb D_t^\alpha$ given by
\begin{align*}
\mathbb D_t^\alpha u(t):=\int_0^t\frac{(t-\tau)^{-\alpha}}{\Gamma(1-\alpha)}u^{\prime}(\tau)\;d\tau.
\end{align*}
If $\alpha=1$ and the function $u$ is smooth enough, one can show that $\mathbb D_t^1u(t)=u'(t)$.

The Mittag-Leffler function with two parameters is defined as follows:
\begin{equation*}
E_{\alpha ,\beta }(z):=\sum_{n=0}^{\infty }\frac{z^{n}}{\Gamma (\alpha
n+\beta )},\;\;\alpha >0,\;\beta \in {\mathbb{C}},\quad z\in {\mathbb{C}}.
\end{equation*}
We have that $E_{\alpha
,\beta }(z)$ is an entire function.
In the literature, the notation $E_{\alpha }=E_{\alpha ,1}$ is frequently
used.
It is well-known that, for $0<\alpha <2$:
\begin{equation}
\mathbb{D}_{t}^{\alpha }E_{\alpha ,1}(zt^{\alpha })=zE_{\alpha
,1}(zt^{\alpha }),\quad t>0,z\in {\mathbb{C}},  \label{Der-ML}
\end{equation}%
namely, for every $z\in {\mathbb{C}}$, the function $u(t):=E_{\alpha
,1}(zt^{\alpha })$ is a solution of the scalar valued ordinary differential
equation:
\begin{equation*}
\mathbb{D}_{t}^{\alpha }u(t)=zu(t),\;\;t>0,\;0<\alpha <2.
\end{equation*}%

Finally, for a real number $0<\alpha<1$, we shall denote by $\Phi_\alpha$ the Wright function defined by
\begin{align}\label{wright}
\Phi_\alpha(z):=\sum_{k=0}^\infty\frac{(-z)^k}{\Gamma(\alpha k+1)k!},\;\;\;z\in\CC.
\end{align}
The following formula on the moments is well-known (see e.g. \cite{KLW1,KLW2}):
\begin{align}\label{moment}
\int_0^\infty\Phi_\alpha(t)t^p\;dt=\frac{\Gamma(p+1)}{\Gamma(\alpha p+1)},\;\;\forall\; p>-1.
\end{align}

 For more details on time fractional derivatives, the Mittag-Leffler and the Wright functions, we
refer the reader to \cite{Go-Lu-Ma99,Go-Ma97,Mai,Go-Ma00,Po99} and the
references therein.


\subsection{Statement of the main results} \label{sec-mr-state}

In this section we state the main results of the paper. Recall that, we consider the following semi-linear space-time fractional order Cauchy problems:
\begin{equation}\label{eq31}
\begin{cases}
\mathbb D_t^\alpha u+(-\Delta)^su=f(u)\;\;&\mbox{ in }\;(0,T)\times\RR,\\
u(0,\cdot)=u^0 &\mbox{ in }\;\RR
\end{cases}
\end{equation}
and
\begin{equation}\label{eq32}
\begin{cases}
\mathbb D_t^\alpha u_h+(-\Delta_h)^su_h=f(u_h)\;\;&\mbox{ in }\;(0,T)\times \ZZ_h,\\
u_h(0,\cdot)=u_h^0&\mbox{ in }\;\ZZ_h,
\end{cases}
\end{equation}
where $0<\alpha\le 1$, $0<s<1$ are real numbers and  $\mathbb D_t^\alpha$ denotes the Caputo time-fractional derivative of order $\alpha$ given in \eqref{CFD}.

We introduce our notion of solution.

\begin{definition}
Let $0<s<1$ and $0<\alpha\le 1$.
\begin{enumerate}
\item By a local strong solution of the system \eqref{eq31}, we mean a function $u^\alpha\in C([0,T_{\max});C_0(\RR))$,  for some $T_{\max}>0$, satisfying the following conditions:
\begin{itemize}
\item $\mathbb D_t^\alpha u^\alpha\in  C((0,T_{\max});C_0(\RR))$;

\item  $u^{\alpha}(t,\cdot)\in D((-\Delta)_c^s)$ for a.e. $t\in (0,T_{\max})$;

\item  the first identity in \eqref{eq31} holds poitwise for a.e. $(t,x)\in (0,T_{\max})\times\RR$, and the initial condition is satisfied.
\end{itemize}

\item By a local strong solution of the system \eqref{eq32}, we mean a function $u_h^\alpha\in C([0,T_{\max});C_0(\ZZ_h))$,  for some $T_{\max}>0$, satisfying the following conditions:

\begin{itemize}
\item  $\mathbb D_t^\alpha u_h^\alpha\in C((0,T_{\max});C_0(\ZZ_h))$;

\item  the first identity in  \eqref{eq32} is satisfied for a.e. $t\in (0,T_{\max})$ and all $jh\in\ZZ_h$,  and the initial condition holds.
\end{itemize}
\end{enumerate}
If $T_{\max}=\infty$, then we say that $u^\alpha$ or $u_h^\alpha$ is a global strong solution.
\end{definition}

We assume the following condition on $f$:
\begin{align}\label{cond-f}
f\in C^1(\RR)\;\mbox{ and }\; f(0)=0.
\end{align}

We mention that, it is easy to see that the assumption  \eqref{cond-f} implies that $f$ satisfies the following local-Lipschitz condition:
\begin{equation}\label{eq-LLC}
|f(\xi_1)-f(\xi_2)|\le M|\xi_1-\xi_2|
\end{equation}
for all $\xi_1, \xi_2\in \{\xi\in \RR:\; |\xi|\le M\}$ where $M>0$ is a fixed constant.

Before stating our main results, for the sake of completeness, we include the result on the existence of solutions to the systems \eqref{eq31} and \eqref{eq32}. We notice that this existence result can be found in \cite{Henry} for the case $\alpha=s=1$ and in \cite[Theorem 4.2.2]{Gal-War} for the general case $0<\alpha\le 1$ and $0<s\le 1$. Since this is not the main concern of the present article, we will not go into details.

\begin{theorem}\label{existence}
Let $0<s<1$, $0<\alpha\le 1$ and assume that $f$ satisfies \eqref{cond-f}.
Then the following assertions hold.
\begin{enumerate}
\item For every $u^0\in C_0(\RR)$, there exists  $T_{\max,1}>0$ such that \eqref{eq31} has a unique local strong solution $u^\alpha\in C([0,T_{\max,1});C_0(\RR))$.

\item For every $u_h^0\in C_0(\ZZ_h)$, there exists  $T_{\max,2}>0$ such that \eqref{eq32} has a unique local strong solution $u_h^\alpha\in C([0,T_{\max,2});C_0(\ZZ_h))$.
\end{enumerate}
\end{theorem}

The following theorem  is the main result of the paper.

\begin{theorem} \label{main-theorem1}
Let $0<s<1$, $0<\alpha\le 1$ be real numbers and assume that $f$ satisfies \eqref{eq-LLC}.
Let $u_h^0\in C_0(\ZZ_h)$ and let $u_h^\alpha$ be the associated unique local strong solution of  \eqref{eq32} on $(0,T_{\max,2})$. Then, there exists $u^0\in C_0(\RR)$ such that the corresponding unique local strong solution $u^\alpha$ on $(0,T_{\max,1})$ of \eqref{eq31} satisfies
\begin{align}\label{eq-inh}
\lim_{h\downarrow 0}\sup_{0\le t\le \mathcal T}\|\mathcal R_hu^\alpha(t)-u_h^\alpha(t)\|_{C_0(\ZZ_h)}=0,
\end{align}
for all $\mathcal T\in [0,T_{\max})$, where $T_{\max}:=\min\{T_{\max,1},T_{\max,2}\}$ and $\mathcal R_h$ is the operator defined in \eqref{mapP}.
\end{theorem}

We conclude this section with the following remark.
\begin{remark}
{\em Regarding the convergence results in Theorem \ref{main-theorem1}, unfortunately we do not know analytically the rate of convergence. But our simulations results obtained in Section \ref{numerics} clearly show that we have a rate of convergence which depends on the parameter $s$.
}
\end{remark}

\section{Preliminaries}\label{sec-pre}

In this section we give some preliminary results that are needed in the proofs of our main results.

\subsection{A version of the Trotter-Kato approximation theorem}

\begin{definition}\label{def-21}
A sequence of Banach spaces $X_n$ ($n\in\NN$) together with a sequence of  bounded linear maps $\Pi_n:X\to X_n$ is said to approximate the Banach space $X$ if
\begin{align*}
\lim_{n\to\infty}\|\Pi_nf\|_{X_n}=\|f\|_X\;\mbox{ for all }\; f\in X.
\end{align*}
\end{definition}

We have the following approximation result whose proof can be found in \cite[Theorem 2.6]{Rosen} (see also \cite{Kur}).

\begin{theorem}\label{thm-Kato}
Suppose $\{X_n\}_{n\ge 1}$ approximates the Banach space $X$ and $T_n(t):X_n\to X_n$, $T(t):X\to X$ are strongly continuous contraction semigroups with generators $A_n$ and $A$, respectively. If for every $f\in D(A)$ (or a core of $A$) there exists a sequence $\{f_n\}_{n\ge 1}\subset D(A_n)$  with the properties
\begin{align}
\lim_{n\to\infty}\|\Pi_nf-f_n\|_{X_n}=0\;\mbox{ and }\; \lim_{n\to\infty}\|\Pi_nAf-A_nf_n\|_{X_n}=0,
\end{align}
then for every $f\in X$ we have
\begin{align}
\lim_{n\to\infty}\sup_{0\le t\le S}\|\Pi_nT(t )f-T_n(t)\Pi_nf\|_{X_n}=0,
\end{align}
for every $S\in [0,\infty)$.
\end{theorem}

\subsection{The continuous and discrete fractional Laplace operators}

Let $0<s<1$ and let  $(-\Delta)^s$ be the fractional Laplace operator defined  in \eqref{FL}.

We also let $(-\Delta)_2^s$ be the selfadjoint  operator on $L^2(\RR)$  with domain
\begin{align*}
D((-\Delta)_2^s):=\{u\in H^s(\RR):\; (-\Delta)^su\in L^2(\RR)\},
\end{align*}
where the fractional order Sobolev space $H^s(\RR)$ is defined by
\begin{align*}
H^s(\RR):=\Big\{u\in L^2(\RR):\;\int_{\RR}\int_{\RR}\frac{|u(x)-u(y)|^2}{|x-y|^{1+2s}}\;dxdy<\infty\Big\}.
\end{align*}
Then, it is nowadays well-known that the operator $A:=-(-\Delta)_2^s$ generates a submarkovian semigroup $(e^{-t(-\Delta)_2^s})_{t\ge 0}$ in $L^2(\RR)$.

Let $(-\Delta)_c^s$ be the part of $(-\Delta)_2^s$ on $C_0(\RR)$. That is,
\begin{align*}
\begin{cases}
D((-\Delta)_c^s):=\Big\{U\in C_0(\RR)\cap D((-\Delta)_2^s):\; (-\Delta)^sU\in C_0(\RR)\Big\},\\
 (-\Delta)_c^sU:=(-\Delta)_2^sU.
 \end{cases}
\end{align*}

We have the following result. We include the proof for the sake of completeness.

\begin{proposition}\label{prop-23}
The operator $-(-\Delta)_c^s$ generates a strongly continuous semigroup  of contractions  $(e^{-t(-\Delta)_c^s})_{t\ge 0}$ on $C_0(\RR)$.
\end{proposition}

\begin{proof}
  Let $F\in L^2(\RR)$,  $\lambda> 0$ a real number, and consider the following Poisson problem:
 \begin{align}\label{PB}
 \lambda U+(-\Delta)^sU=F\;\mbox{ in }\;\RR.
\end{align}
 By a weak solution of \eqref{PB}, we mean a function $U\in H^s(\RR)$ such that the identity
 \begin{align*}
 \lambda\int_{\RR}UV\;dx+\frac{C_s}{2}\int_{\RR}\int_{\RR}\frac{(U(x)-U(y))(V(x)-V(y))}{|x-y|^{1+2s}}\;dxdy=\int_{\RR}FV\;dx,
 \end{align*}
holds for every $V\in H^s(\RR)$.
Firstly, we notice that for every $F\in L^2(\RR)$, the problem \eqref{PB} has a unique weak solution. In addition, if $F\in C_0(\RR)$, then $U\in C^{0,s}(\RR)\cap C_0(\RR)$ (see e.g. \cite{Grub,RS-DP2} and their references). This shows that the resolvent $(\lambda+(-\Delta)^s)^{-1}$ leaves the space $C_0(\RR)$ invariant. Using semigroups theory, the above property implies that the operator $e^{-t(-\Delta)_2^s}$ also leaves the space $C_0(\RR)$ invariant for every $t\ge 0$. Thus, $-(-\Delta)_c^s$ generates a semigroup of contractions on $C_0(\RR)$. The strong continuity of the semigroup follows from the fact that $D((-\Delta)_c^s)$ is dense in $C_0(\RR)$ and the proof is finished.
\end{proof}

Next, recall that we have defined
\begin{align*}
C_0(\ZZ_h):=\Big\{f:\ZZ_h\to\RR\;\mbox{ with }\;\lim_{|jh|\to\infty}f(jh)=0\Big\}.
\end{align*}
Let  $(-\Delta_h)^s$ be the operator defined in \eqref{dFL}.
Then, we have following result.

\begin{proposition}\label{Prop-SG-h}
The bounded operator $A_h:=-(-\Delta_h)^s$ (that is, with domain $D(A_h)=C_0(\ZZ_h)$)
generates a uniformly strongly continuous semigroup of contractions $T_h=(e^{-t(-\Delta_h)^s})_{t\ge 0}$ on $C_0(\ZZ_h)$ and is given explicitly for $f\in C_0(\ZZ_h)$ by
\begin{equation}\label{SG-h}
T_h(t)f(jh) =\sum_{n\in\ZZ} L^s_{n-j}(t/h^{2s}) f(nh),
\end{equation}
where 
$$
L^s_n(t) := \frac{1}{2\pi} \int_{-\pi}^{\pi} e^{-t(4\sin^2 \theta/2)^s} e^{-in\theta} d\theta = \int_0^{\infty} e^{-2\lambda}I_n(2\lambda) f_{x,s}(\lambda)d\lambda,
$$
being $f_{x,s}(\lambda)$ the L\'evy function.
\end{proposition}
The proof follows from \cite[p.1371, Theorem 1.3]{LR18}.
\begin{remark}
{\em
We notice that $(T_h(t))_{t\geq 0}$ is an uniformly continuous semigroup of contractions on $C_0(\ZZ_h)$ follows from the positivity of the Bessel and L\'evy functions $I_k, f_{x,s}$ and \cite[Theorem 1.3, item (v)]{LR18} which implies that the estimate
\begin{align*}
\|T_h(t)f\|_{C_0(\ZZ_h)} \leq \|f\|_{C_0(\ZZ_h)} \int_0^{\infty} \left[ \sum_{n\in \ZZ} e^{-\frac{2t}{h^{2s}}}I_{n}  \left(\frac{2t}{h^{2s}}\right)\right]f_{t,s}(\lambda)d\lambda = \|f\|_{C_0(\ZZ_h)},
\end{align*}
holds for all $f \in C_0(\ZZ_h).$ 
}
\end{remark}

The following convergence result  taken from \cite[Theorem 1.7]{Setal} will be crucial in the proof of our main results.

\begin{theorem}\label{thm-Stinga}
Let $0<\gamma\le 1$ and $0<s<1$. Then the following assertions holds.
\begin{enumerate}
\item If $U\in C^{0,\gamma}(\RR)$ and $2s<\gamma$, then there is a constant $C>0$ (independent of $U$ and $h$) such that
\begin{align*}
\|(-\Delta_h)^s(\mathcal R_hU)-\mathcal R_h((-\Delta)^sU)\|_{C_0(\ZZ_h)}\le C[U]_{C^{0,\gamma}}h^{\gamma-2s}.
\end{align*}

\item If $U\in C^{1,\gamma}(\RR)$ and $\gamma<2s<1+\gamma$, then there is a constant $C>0$ (independent of $U$ and $h$) such that
\begin{align*}
\|(-\Delta_h)^s(\mathcal R_hU)-\mathcal R_h((-\Delta)^sU)\|_{C_0(\ZZ_h)}\le C[U^\prime]_{C^{0,\gamma}}h^{\gamma-2s+1}.
\end{align*}
\end{enumerate}
\end{theorem}

\section{Proof of the main result}\label{sec-pro}

 In this section we give the proof of our main result, namely, Theorem \ref{main-theorem1}.

\begin{proof}[\bf Proof of Theorem \ref{main-theorem1}]
We prove the result in several steps. Throughout the proof if $v_h$, $w_h\in C_0(\ZZ_h)$ are such that $\lim_{h\downarrow 0}\|v_h-w_h\|_{C_0(\ZZ_h)}=0$, then we shall sometimes write $v_h=w_h+o(h)$.\\

{\bf Step 1}: Let $\mathcal R_h$ be the mapping defined in \eqref{mapP}. We claim that $\mathcal R_h$ and $C_0(\ZZ_h)$ approximate the Banach space $C_0(\RR)$ in the sense of Definition \ref{def-21}. Indeed, firstly it is clear that $\|\mathcal R_hF\|_{C_0(\ZZ_h)}\le \|F\|_{C_0(\RR)}$ for every $F\in C_0(\RR)$. Secondly, let $h>0$ be a fixed real number. Then, for every $x\in \RR$, there exists $j\in \ZZ$ such that $hj\le x< (j+1)h$. This implies that $|x-hj|\to 0$ as $h\to 0$.  This fact, together with the previous observation imply that $\lim_{h\to 0}\|\mathcal R_hF\|_h=\|F\|_{C_0(\RR)}$ for the every $F\in C_0(\RR)$ and the claim is proved.

We also notice that, since $\mathcal R_h$ and  $C_0(\ZZ_h)$ approximate  $C_0(\RR)$, it follows that
for every  $g_h\in C_0(\ZZ_h)$, there exists a function $g\in C_0(\RR)$ such that
\begin{equation}\label{lim}
 \lim_{h\downarrow 0}\|\mathcal R_hg-g_h\|_{C_0(\ZZ_h)}=0.
\end{equation}

{\bf Step 2}:  Let $T=(e^{-t(-\Delta)_c^s})_{t\ge 0}$ and $T_h:=(e^{-t(-\Delta_h)^s})_{t\ge 0}$ be the strongly continuous and contractive semigroups  given in Propositions \ref{prop-23} and \ref{Prop-SG-h}, respectively.
We claim that for every $u^0\in C_0(\RR)$ and $\mathcal T>0$, we have
\begin{align}\label{semi-ine}
\lim_{h\downarrow 0}\sup_{0\le t\le \mathcal T} \|\mathcal R_h(T(t)u^0)-T_h(t)\mathcal R_hu^0\|_{C_0(\ZZ_h)}=0.
\end{align}

It suffices to prove  \eqref{semi-ine} for every $u_0$ in a dense subspace of $C_0(\RR)$. Indeed, let $u_0\in D((-\Delta)_c^s)\cap C_c^{1,1}(\RR)$.  It follows from Theorem \ref{thm-Stinga} that
\begin{align}\label{in-ST}
\lim_{h\downarrow 0}\| (-\Delta_h)^s\mathcal R_hu_0-\mathcal R_h((-\Delta)^su_0)\|_{C_0(\ZZ_h)}=0.
\end{align}
Using Step 1, the convergence in \eqref{in-ST} and applying Theorem \ref{thm-Kato}, we can deduce that \eqref{semi-ine} holds.\\

{\bf Step 3}: We prove  \eqref{eq-inh} for the case $\alpha=1$.
Recall that by Theorem \ref{existence}, under the assumption \eqref{eq-LLC} on the nonlinearity $f$,
there exist  local strong solutions $u^1\in C([0,T_{\max,1}); C_0(\RR))$ and $u_h^1\in C([0,T_{\max,2})); C_0(\ZZ_h))$, for some $T_{\max,1}>0$ and $T_{\max,2}>0$.  In addition, using semigroups theory, we have that
\begin{align}\label{SG1}
u^1(t)=T(t)u^0+\int_0^tT(t-\tau)f(u^1(\tau))\;d\tau,
\end{align}
for every $t\in [0,T_{\max,1})$ and
\begin{align}\label{SG2}
u_h^1(t)=T_h(t)u_h^0+\int_0^tT_h(t-\tau)f(u_h^1(\tau))\;d\tau,
\end{align}
for every $t\in [0,T_{\max,2})$.

Next, let $0\le \mathcal T<T_{\max}:=\min\{T_{\max,1},T_{\max,2}\}$ be an arbitrary real number. Let $u_h^0\in C_0(\ZZ_h)$ be fixed. Choose a function $u^0\in C_0(\RR)$ satisfying \eqref{lim}. Then, using \eqref{semi-ine} and the representations \eqref{SG1}-\eqref{SG2}, we get that for every $t\in [0,\mathcal T]$,
\begin{align}\label{Wm}
\mathcal R_hu^1(t)=&\mathcal R_h(T(t)u^0)+\int_0^t\mathcal R_h\Big(T(t-\tau)f(u^1(\tau))\Big)\;d\tau\notag\\
=&\Big(\mathcal R_h(T(t)u^0)-T_h(t)\mathcal R_hu^0\Big) +T_h(t)\mathcal R_hu^0+\int_0^t\mathcal R_h\Big(T(t-\tau)f(u^1(\tau))\Big)\;d\tau\notag\\
=&o(h)+T_h(t)\mathcal R_hu^0+\int_0^t\Big(T_h(t-\tau)\mathcal R_hf(u^1(\tau))+o(h)\Big)\;d\tau\notag\\
=&o(h)+T_h(t)\mathcal R_hu^0 +\int_0^tT_h(t-\tau)f(\mathcal R_hu^1(\tau))\;d\tau+to(h).
\end{align}
It follows from \eqref{Wm} that
\begin{align}\label{AAA}
\mathcal R_hu^1(t)-u_h^1(t)=& (1+t)o(h)+T_h(t)\mathcal R_hu^0-T_h(t)u_h^0\notag\\
&+\int_0^tT_h(t-\tau)\Big(f(\mathcal R_hu^1(\tau))-f(u_h^1(\tau))\Big)\;d\tau\notag\\
=&(2+t)o(h) +\int_0^tT_h(t-\tau)\Big(f(\mathcal R_hu^1(\tau))-f(u_h^1(\tau))\Big)\;d\tau,
\end{align}
where we have used the fact that
\begin{align*}
\|T_h(t)\mathcal R_hu^0-T_h(t)u_h^0\|_{C_0(\ZZ_h)}=\|T_h(t)\Big(\mathcal R_hu^0-u_h^0\Big)\|_{C_0(\ZZ_h)}\le \|\mathcal R_hu^0-u_h^0\|_{C_0(\ZZ_h)},
\end{align*}
which converges to zero as $h\downarrow 0$ by \eqref{lim}.

Now set
\begin{align*}
\Psi_h(t):=\|\mathcal R_hu^1(t)-u_h^1(t)\|_{C_0(\ZZ_h)},\;\;0\le t\le\mathcal T.
\end{align*}
We notice that
\begin{align*}
\Psi_h(0)=o(h).
\end{align*}
Using the fact that $\|T_h(t)g_h\|_h\le \|g_h\|_h$ for every $g_h\in C_0(\ZZ_h)$ and the local Lipschitz continuity assumption \eqref{eq-LLC} on $f$, we can deduce from \eqref{AAA} that there is a constant $K>0$ (depending only on $\mathcal T$ and the Lipschitz constant $M$) such that
\begin{align*}
\Psi_h(t)\le Kh+K\int_0^t\Psi_h(\tau)\;d\tau,\;\;\;0\le t\le\mathcal T.
\end{align*}
Using Gronwall's inequality, the above estimate implies that
\begin{align*}
\Psi_h(t) \leq Kh e^{Kt}\;\mbox{ for all }\; t\in  [0,\mathcal T].
\end{align*}
We have shown \eqref{eq-inh} for $\alpha=1$. \\

{\bf Step 4}: Next, we prove \eqref{eq-inh} for the case $0<\alpha<1$.
Firstly, we recall that under the assumption \eqref{eq-LLC} on $f$, Theorem \ref{existence} implies
the existence of local strong solutions $u^\alpha\in C([0,T_{\max,1}); C_0(\RR))$ and $u_h^\alpha\in C([0,T_{\max,2}); C_0(\ZZ_h))$, for some $T_{\max,1}>0$ and $T_{\max,2}>0$. In addition, using the theory of fractional order Cauchy problems (see e.g. \cite{Bah, KLW1,KLW2,KLW3}), we have  that
\begin{align}\label{u}
u^\alpha(t)=\mathbb S^\alpha(t)u^0+\int_0^t\mathbb P^\alpha(t-\tau)f(u^\alpha(\tau))\;d\tau
\end{align}
for every $0\le t<T_{\max,1}$, and
\begin{align}\label{uh}
u_h^\alpha(t)=\mathbb S^\alpha_h(t)u_h^0+\int_0^t\mathbb P^\alpha_h(t-\tau)f(u_h^\alpha(\tau))\;d\tau,
\end{align}
for every $0\le t<T_{\max,2}$, where the operators $\mathbb S^\alpha(t)$ (for every $t\ge 0$) and  $\mathbb P^\alpha(t)$ (for every $t>0$) are given for every $g\in C_0(\RR)$  by
\begin{align}\label{SS1}
\mathbb S^\alpha(t)g=\int_0^\infty\Phi_\alpha(\tau)T(\tau t^\alpha)g\;d\tau,\;\;\mathbb P^\alpha(t)g=\alpha t^{\alpha-1}\int_0^\infty\tau\Phi_\alpha(\tau)T(\tau t^\alpha)g\;d\tau,
\end{align}
and $\mathbb S_h^\alpha(t)$ (for every $t\ge 0$) and $\mathbb P_h^\alpha(t)$ (for every $t> 0$) are given for $g_h\in C_0(\ZZ_h)$ by
\begin{align}\label{PP1}
\mathbb S_h^\alpha(t)g_h=\int_0^\infty\Phi_\alpha(\tau)T_h(\tau t^\alpha)g_h\;d\tau,\;\;\mathbb P_h^\alpha(t)g_h=\alpha t^{\alpha-1}\int_0^\infty\tau\Phi_\alpha(\tau)T_h(\tau t^\alpha)g_h\;d\tau,
\end{align}
and we recall that $\Phi_{\alpha}$ is the Wright function defined in \eqref{wright}.\\

{\bf Step 5}: Let $u_h^0\in C_0(\ZZ_h)$ and $u^0\in C_0(\RR)$ satisfy \eqref{lim}. We claim that for every $\mathcal T>0$, we have
\begin{align}\label{claim1}
\lim_{h\downarrow 0}\sup_{0\le t\le \mathcal T}\|\mathcal R_h(\mathbb S^\alpha(t)u^0)-\mathbb S_h^\alpha(t)u_h^0\|_{C_0(\ZZ_h)}=0.
\end{align}
Indeed, using the representations \eqref{SS1}-\eqref{PP1} and Step 2, we can deduce that for every $t\ge 0$ we have
\begin{align}\label{Wm2}
\mathcal R_h(\mathbb S^\alpha(t)u^0)=&\int_0^\infty\Phi_\alpha(\tau)\mathcal R_h(T(\tau t^\alpha)u_0)\;d\tau\notag\\
=&\int_0^\infty\Phi_\alpha(\tau)\Big(T_h(\tau t^\alpha)\mathcal R_hu_0+o(h)\Big)\;d\tau\notag\\
=&\int_0^\infty\Phi_\alpha(\tau)T_h(\tau t^\alpha)\mathcal R_hu_0\;d\tau +o(h)\notag\\
=&\mathbb S_h^\alpha(t)\mathcal R_hu^0 +o(h),
\end{align}
where in the third equality we have used \eqref{moment}. It follows from \eqref{Wm2} that

\begin{align*}
\mathcal R_h(\mathbb S^\alpha(t)u^0)-\mathbb S_h^\alpha(t)u_h^0=\mathbb S_h^\alpha(t)\mathcal R_hu^0 -\mathbb S_h^\alpha(t)u_h^0 +o(h)=o(h),
\end{align*}
where we have also used that
\begin{align*}
\|\mathbb S_h^\alpha(t)\mathcal R_hu^0 -\mathbb S_h^\alpha(t)u_h^0 \|_{C_0(\ZZ_h)}&\le \int_0^\infty\Phi_\alpha(\tau)\|T_h(\tau t^\alpha)(\mathcal R_hu^0-u_h^0)\|_{C_0(\ZZ_h)}\;d\tau\\
&\le \int_0^\infty\Phi_\alpha(\tau)\|\mathcal R_hu^0-u_h^0\|_{C_0(\ZZ_h)}\;d\tau\\
&\le \|\mathcal R_hu^0-u_h^0\|_{C_0(\ZZ_h)}=o(h),
\end{align*}
which follows from \eqref{lim}. The proof of the claim \eqref{claim1} is finished.\\

{\bf Step 6:} Let $u_h^0\in C_0(\ZZ_h)$ and $u^0\in C_0(\RR)$ satisfy \eqref{lim}.  We claim that for every $ \mathcal T>0$, we have
\begin{align}\label{claim2}
\lim_{h\downarrow 0}\sup_{0\le t\le \mathcal T}\Big(t^{1-\alpha}\|\mathcal R_h(\mathbb P^\alpha(t)u^0)-\mathbb P_h^\alpha(t)u_h^0\|_{C_0(\ZZ_h)}\Big)=0.
\end{align}
Using the representations \eqref{SS1}-\eqref{PP1} and proceeding as in the proof of  \eqref{claim1}, we easily get that
\begin{align*}
t^{1-\alpha}&\|\mathcal R_h(\mathbb P^\alpha(t)u^0)-\mathbb P_h^\alpha (t)u_h^0\|_{C_0(\ZZ_h)}\\
\le&\alpha \int_0^\infty\tau\Phi_\alpha(\tau)\|\mathcal R_h(T(\tau t^\alpha)u^0)-T_h(\tau t^\alpha) u_h^0\|_{C_0(\ZZ_h)}\\
\le &\alpha \cdot o(h)\int_0^\infty\tau\Phi_\alpha(\tau)\;d\tau
=\frac{\alpha\Gamma(2)}{\Gamma(\alpha+1)}o(h),
\end{align*}
where we have also used \eqref{moment}.
The proof of the claim \eqref{claim2} is complete.\\

{\bf Step 7}: Finally, set $T_{\max}:=\min\{T_{\max,1},T_{\max,2}\}$, and let $0\le \mathcal T<T_{\max}$. Then, using \eqref{claim1}, \eqref{claim2} and the representations \eqref{u}-\eqref{uh},  we get that for every $0\le t\le\mathcal T$,
\begin{align*}
\mathcal R_hu^\alpha(t)=&\mathcal R_h(\mathbb S^\alpha(t)u^0)+\int_0^t\mathcal R_h(\mathbb P^\alpha(t-\tau)f(u^\alpha(\tau)))\;d\tau\\
=&\mathbb S^\alpha_h(t)u_h^0 +o(h)+\int_0^t\Big(\mathbb P^\alpha_h(t-\tau)\mathcal R_hf(u^\alpha(\tau))+(t-\tau)^{\alpha-1}o(h)\Bigg)\;d\tau\\
=&\mathbb S^\alpha_h(t)u_h^0 +o(h)+\int_0^t\mathbb P^\alpha_h(t-\tau)f(\mathcal R_hu^\alpha(\tau))\;d\tau+t^\alpha o(h).
\end{align*}
This identity implies that
\begin{align}\label{EST}
\mathcal R_hu^\alpha(t)-u_h^\alpha(t)= (1+t^\alpha) o(h)+\int_0^t\mathbb P^\alpha_h(t-\tau)\Big(f(\mathcal R_hu^\alpha(\tau))-f(u_h^\alpha(\tau))\Big)\;d\tau,
\end{align}
for every $0\le t\le\mathcal T$.

Let
\begin{align*}
\Psi_h(t):=\|\mathcal R_hu^\alpha(t)-u_h^\alpha(t)\|_{C_0(\ZZ_h)},\;\;0\le t\le\mathcal T.
\end{align*}
We notice that
\begin{align*}
\Psi_h(0)=o(h).
\end{align*}
Using \eqref{PP1}, the fact that $\|T_h(t)g_h\|_h\le\|g_h\|_h$ for every $g_h\in C_0(\ZZ_h)$ and proceeding as in Step 6, we can easily deduce that for every $t>0$, we have 
\begin{align}\label{est-GW}
\|\mathbb P^\alpha_h(t)g_h\|_{C_0(\ZZ_h)}\le  t^{\alpha-1}\|g_h\|_{C_0(\ZZ_h)}.
\end{align}
Using \eqref{est-GW} and the assumption \eqref{eq-LLC} on the nonlinearity $f$, we get from \eqref{EST} that there is a constant $K>0$ (depending only on $\mathcal T$ and the Lipschitz constant $M$) such that
\begin{align}\label{mw}
\Psi_h(t)\le Kh+K\int_0^t(t-\tau)^{\alpha-1}\Psi_h(\tau)\;d\tau,\;\;\;\;0\le t\le\mathcal T.
\end{align}
Using the generalized Gronwall's inequality (see e.g. \cite[Theorem 1]{YGY}), we immediately deduce from \eqref{mw} that  for all $t\in  [0,\mathcal T]$, we have
\begin{align}\label{CL}
\Psi_h(t) \leq &Kh \Big(1+\int_0^t\sum_{n=1}^\infty\frac{(\Gamma(\alpha))^n}{\Gamma(\alpha n)}(t-\tau)^{\alpha n-1}\;d\tau\Big)\notag\\
=&Kh \Big(1+\sum_{n=1}^\infty\frac{(\Gamma(\alpha))^n}{\Gamma(\alpha n+1)}t^{\alpha n}\Big)=Kh E_{\alpha,1} (\Gamma(\alpha)t^\alpha).
\end{align}
The estimate \eqref{CL} together with the fact that the function $E_{\alpha,1}$ is monotone on $[0,\infty)$ yield
\begin{align*}
\lim_{h\downarrow 0}\sup_{0\le t\le \mathcal T}\Psi_h(t)=0.
\end{align*}
We have shown \eqref{eq-inh} for $0<\alpha<1$. The proof of the theorem is finished.
\end{proof}

We conclude this section with the following observation.

\begin{remark}
{\em It follows from the proof of Theorem \ref{main-theorem1}, that in particular,
for every $u_0\in C_0(\RR)$, taking $u_h^0:=\mathcal R_hu^0$, we have that the corresponding local strong solutions $u^\alpha$ and $u_h^\alpha$ satisfy
\begin{align*}
\lim_{h\downarrow 0}\sup_{0\le t\le \mathcal T}\|\mathcal R_hu^\alpha(t)-u_h^\alpha(t)\|_{C_0(\ZZ_h)}=0,
\end{align*}
for every $\mathcal T\in [0,T_{\max})$, where we recall that $T_{\max}:=\min\{T_{\max,1},T_{\max,2}\}$.
}
\end{remark}


\section{Numerical Examples}\label{numerics}
The main goal of this section is to numerically study the validity of \eqref{eq-inh} with the help of two examples. We focus on the case when $s \in (0,1)$ but $\alpha =1$, i.e., the fractional Laplacian with standard time derivative. Moreover, we consider the linear case, however our  implementation directly extends to the semi-linear case after applying Newton type methods. In our first example the solution $u$ is supported in $(0,T) \times \mathbb{R}$ with $T < +\infty$ and in our second example $u$ is compactly supported in $[0,T] \times [-1,1]$ with $T < +\infty$. In both cases, we have set $T = 1$. Before we move on to our examples, we emphasize that numerically we
find that the formulation \eqref{eq:DeltahK} is more suitable than \eqref{eq:DeltahR} and therefore all our
examples use \eqref{eq:DeltahK} to calculate the discrete fractional Laplacian.

\subsection{Example 1}

Let the exact solution to \eqref{EE1} be the function
    \[
        u(t,x) := \left(1+x^2\right)^{-\left(\frac12-s\right)} e^{-t}
    \]
supported on $(0,1)\times \mathbb{R}$. Substituting this expression of $u$ in \eqref{EE1},
we obtain that
    \[
        f(t,x) := \partial_t u + (-\Delta)^s u
               = -u(t,x) + \frac{4^s \Gamma(1/2+s)}{\Gamma(1/2-s)} \left(1+x^2\right)^{-\left(\frac12+s\right)} e^{-t}  .
    \]
With this value of $f$ as the right-hand-side, we next numerically solve \eqref{EE2}
for $u_h$. Here we calculate $(-\Delta_h)^s$ using the formula \eqref{eq:DeltahK}. We apply
Backward Euler discretization in time with a time step $\delta t = 10^{-3}$. In
Figure~\ref{f:1}, we study the norm in \eqref{eq-inh} as $h$ decreases. Notice that
our problem is posed on the unbounded domain $(0,1) \times \mathbb{R}$, to make it
numerically tractable we instead solve \eqref{EE2} on a bounded domain
$(a,b) \subset \mathbb{R}$ with $a = -10^3$ and $b = 10^3$. We then measure the
error in a subdomain $(-10^2,10^2) \subset (a,b)$.
 
\begin{figure}[H]
\centering
    \includegraphics[width=0.5\textwidth]{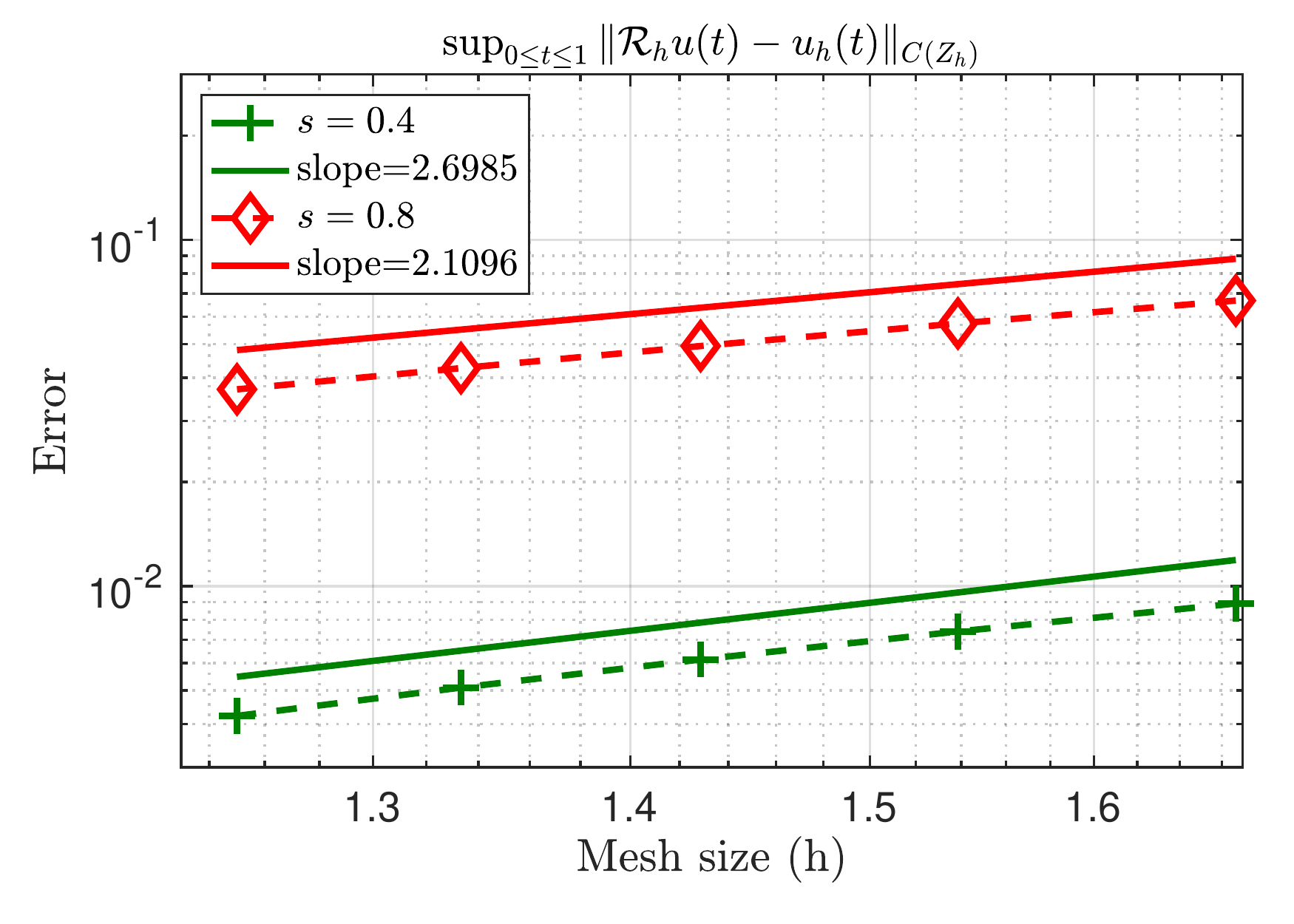}
    \caption{\label{f:1}Example 1: The figure illustrates the error \eqref{eq-inh} as we refine $h$.
    The results are in accordance to Theorem~\ref{main-theorem1}, i.e., the required error goes to
    zero as $h\downarrow 0$. Additionally, we obtain a rate of convergence which appears to be
    $s$-dependent.}
\end{figure}

Figure~\ref{f:1} clearly supports the result of Theorem~\ref{main-theorem1}, i.e.,
the required error goes to zero as $h\downarrow 0$. In addition, we seem to observe a
$s$-dependent rate of convergence. We emphasize that the theoretical justification
of such a rate of convergence is still an open question.

Finally, Figure~\ref{f:comp1} shows a visual comparison between the computed and
the exact solutions at time instance $t = 0.5$ for $s = 0.4$ and $s = 0.8$, respectively.
\begin{figure}[H]
\centering
    \includegraphics[width=0.45\textwidth]{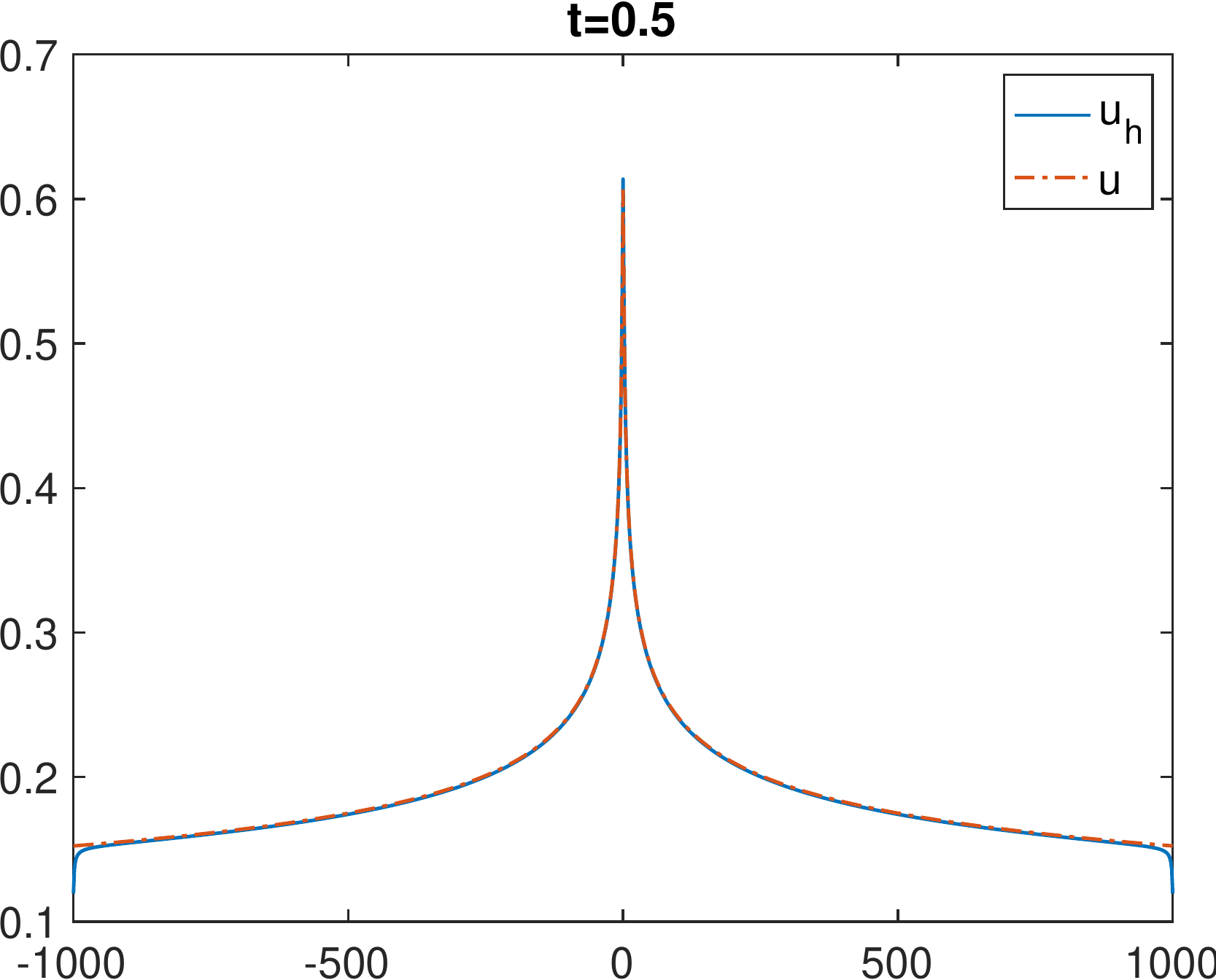}
    \includegraphics[width=0.45\textwidth]{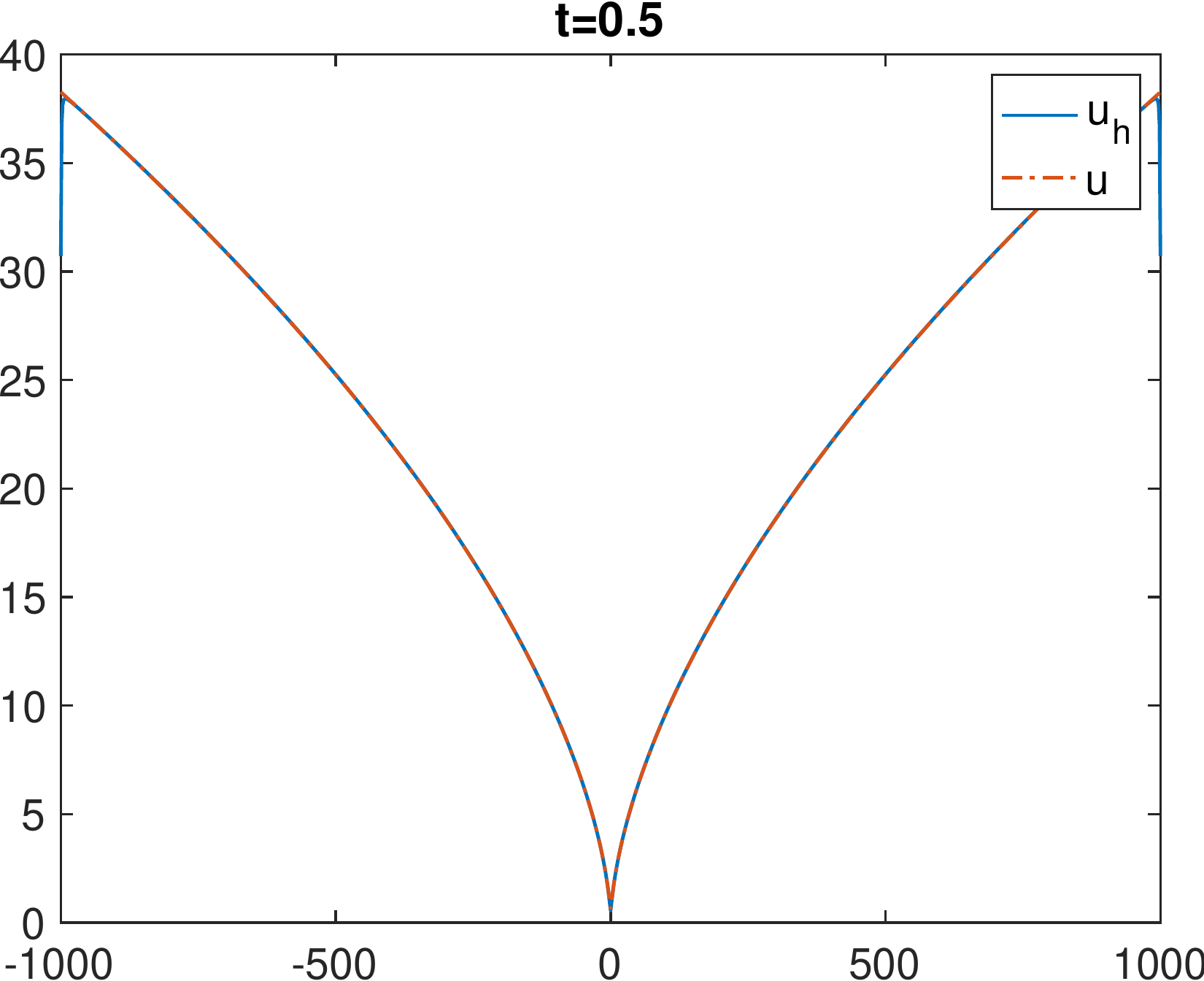}
    \caption{\label{f:comp1}Example 1: A visual comparison between $u$ and $u_h$ for $h=1.6667$
    at $t = 0.5$. The left and right panels corresponds to $s = 0.4$ and $s = 0.8$, respectively.}
\end{figure}

\subsection{Example 2}

In our second example, we consider the following nonsmooth exact solution to
\eqref{EE1}:
    \[
        u(t,x) := \frac{2^{-2s}}{\Gamma(1+s)^2}\left(1-|x|^2\right)^s_{+} e^{-t} .
    \]
Notice that, unlike the previous example, $u$ in this case is compactly
supported in $[0,1] \times [-1,1]$. Next, we define the right-hand-side $f$
as
    \[
        f(t,x) := \partial_t u + (-\Delta)^s u
               = -u(t,x) + e^{-t} .
    \]
Similarly to the previous example, we use this $f$ and solve \eqref{EE2} on
$(0,1) \times (-1,1)$. Again, we employ Backward Euler discretization in
time with a step size $\delta t = 10^{-3}$. Figure~\ref{f:2} (left) shows
the norm in \eqref{eq-inh} as $h$ decreases.
As in the previous example, we observe that this example again supports the result
of Theorem~\ref{main-theorem1}, i.e., the required error goes to zero as $h\downarrow 0$.
In addition, we observe a certain $s$-dependent rate of convergence.
The right panel in Figure~\ref{f:2} shows a visual comparison between $u$ and $u_h$ for $s=0.1$ at
time instance $t = 0.5$.

\begin{figure}[H]
\centering
    \includegraphics[width=0.5\textwidth]{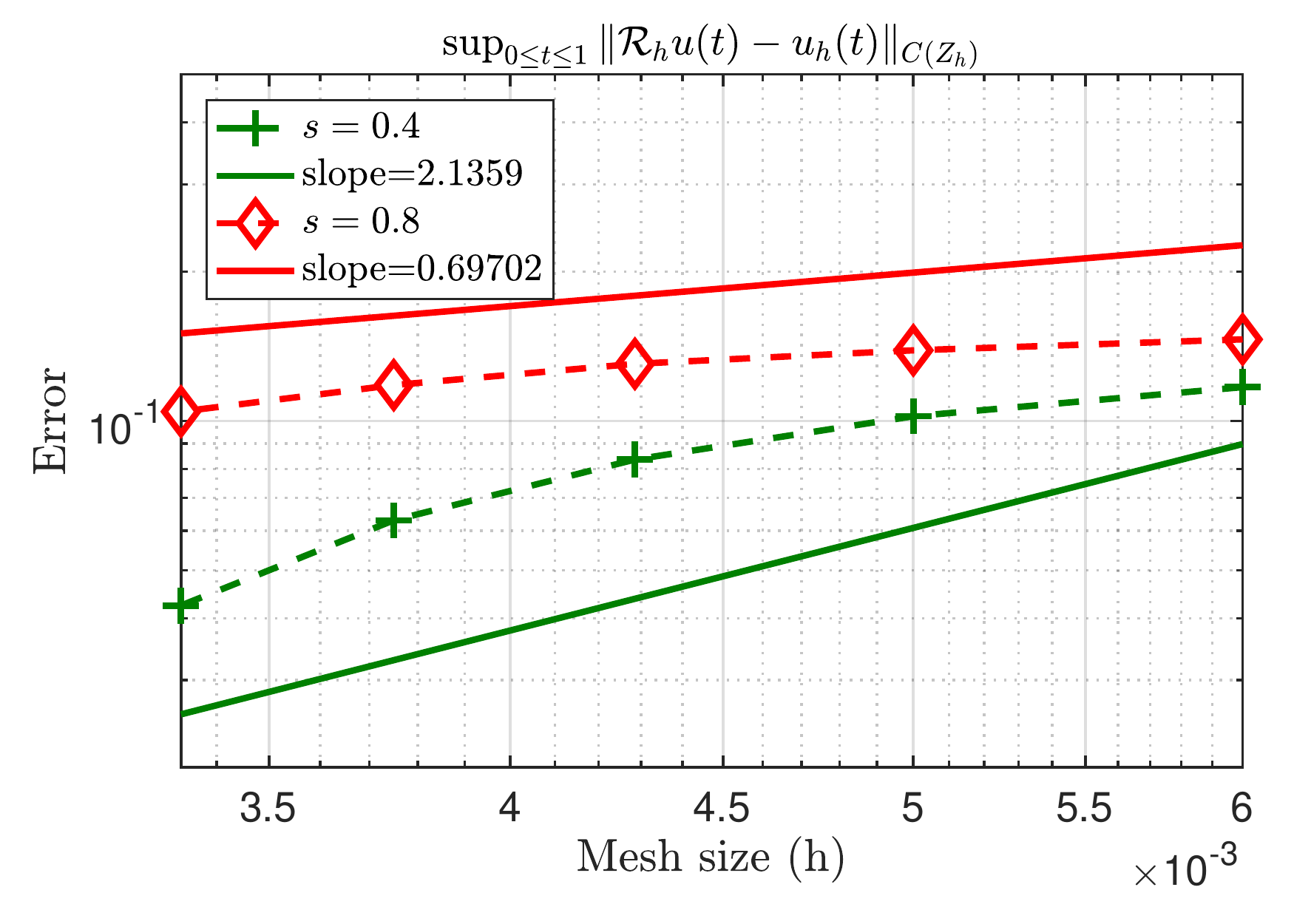}
    \includegraphics[width=0.45\textwidth]{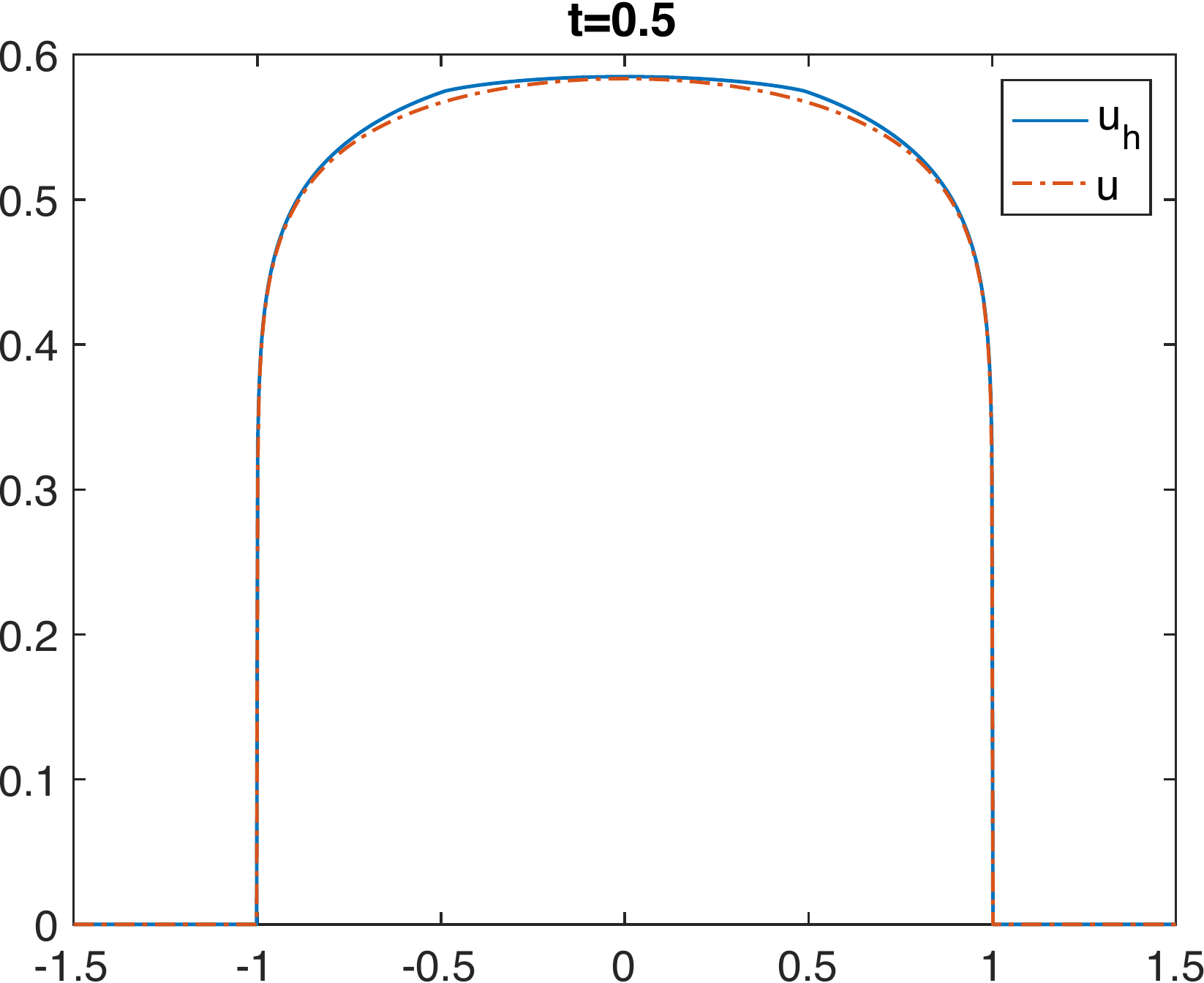}
    \caption{\label{f:2}Example 2: The left panel illustrates the error \eqref{eq-inh} as we refine $h$ and it
    confirms our theoretical findings in Theorem~\ref{main-theorem1}, i.e., the error goes to zero as $h\downarrow 0$.
    Additionally, we obtain a rate of convergence which appears to be $s$-dependent. The right-panel illustrates a
    visual comparison for $s=0.1$, $h=0.003$ at $t = 0.5$.}
\end{figure}

\bibliographystyle{plain}
\bibliography{refs}

\end{document}